\documentclass[12pt,twoside]{amsart}

\usepackage{microtype}
\usepackage[OT1]{fontenc}
\usepackage{textcomp}
\usepackage{type1cm}
\usepackage{amssymb}
\usepackage{mathtools}
\usepackage{esint}
\usepackage{nicefrac}
\usepackage{enumerate}
\usepackage[left=3.2cm,top=3.8cm,right=3.2cm]{geometry}

\usepackage[active]{srcltx}
\usepackage{verbatim}

\usepackage[pdftex]{graphicx}
\usepackage{caption}
\usepackage{subfigure}

\usepackage{xcolor}
\usepackage[pagebackref]{hyperref}
\hypersetup{
    colorlinks,
    linkcolor={red!60!black},
    citecolor={blue!60!black},
   urlcolor={blue!90!black}
}

\numberwithin{equation}{section}

\theoremstyle{plain}
\newtheorem{theorem}{Theorem}[section]

\newtheorem{corollary}[theorem]{Corollary}
\newtheorem{proposition}[theorem]{Proposition}
\newtheorem{lemma}[theorem]{Lemma}

\theoremstyle{remark}
\newtheorem{remark}[theorem]{Remark}

\newtheorem{example}[theorem]{Example}

\theoremstyle{definition}

\newenvironment{labeledlist}[2][\unskip]
{ \renewcommand{\theenumi}{\textnormal{({#2}\arabic{enumi}{#1})}}
  \renewcommand{\labelenumi}{\textnormal{({#2}\arabic{enumi}{#1})}}
  \begin{enumerate} }
{ \end{enumerate} }

\newcommand{\PP}{\mathcal{P}}

\newcommand{\DD}{\mathcal{D}}

\newcommand{\SSS}{\mathcal{S}}
\newcommand{\R}{\mathbb{R}}

\newcommand{\Z}{\mathbb{Z}}
\newcommand{\N}{\mathbb{N}}

\newcommand{\iii}{\mathtt{i}}
\newcommand{\jjj}{\mathtt{j}}

\newcommand{\fii}{\varphi}
\newcommand{\roo}{\varrho}

\newcommand{\lalpha}{\underline{\alpha}}

\renewcommand{\epsilon}{\varepsilon}
\newcommand{\eps}{\varepsilon}
\newcommand{\e}{\varepsilon}

\newcommand{\kk}{\mathtt{k}}

\newcommand{\OO}{\mathcal{O}}

\newcommand{\ii}{\mathtt{i}}

\newcommand{\jj}{\mathtt{j}}

\newcommand{\iin}[1]{\ii|_{#1}}

\DeclareMathOperator{\udimb}{\overline{dim}_B}

\DeclareMathOperator{\diml}{\underline{dim}_L}

\DeclareMathOperator{\dist}{dist}
\DeclareMathOperator{\diam}{diam}

\DeclareMathOperator{\spt}{spt}

\begin{document}

\title{On measures that improve $L^q$ dimension under convolution}

\author{Eino Rossi}
\address{
        Department of Mathematics and Statistics, University of Helsinki \\
        P.O. Box 68  (Pietari Kalmin katu 5) \\
        00014 University of Helsinki, Finland}
\email{eino.rossi@gmail.com}
\author{Pablo Shmerkin}
\address{
        Department of Mathematics and Statistics, Torcuato di Tella University \\
        and CONICET \\
        Av. Figueroa Alcorta 7350 (C1428BCW) \\
        Buenos Aires \\
        Argentina}
\email{pshmerkin@utdt.edu}
\thanks{ER acknowledges the supports of CONICET, the Finnish Academy of Science and Letters, Mittag-Leffler institute, and the University of Helsinki via the project
Quantitative rectifiability of sets and measures in Euclidean Spaces and Heisenberg groups (project No.7516125)}
\thanks{PS was partially supported by Projects CONICET-PIP 11220150100355 and PICT 2015-3675 (ANPCyT)}
\date{\today}

\begin{abstract}
 The $L^q$ dimensions, for $1<q<\infty$, quantify the degree of smoothness of a measure. We study the following problem on the real line: when does the $L^q$ dimension improve under convolution? This can be seen as a variant of the well-known $L^p$-improving property. Our main result asserts that uniformly perfect measures (which include Ahlfors-regular measures as a proper subset) have the property that convolving with them results in a strict increase of the $L^q$ dimension. We also study the case $q=\infty$, which corresponds to the supremum of the Frostman exponents of the measure. We obtain consequences for repeated convolutions and for the box dimension of sumsets. Our results are derived from an inverse theorem for the $L^q$ norms of convolutions due to the second author.
\end{abstract}

\maketitle

\section{Introduction}

Let $\nu$ be a Borel probability measure on $\mathbb{R}^d$. Since convolution (of a function, measure, etc) with $\nu$ is a smoothing operation, a natural problem arises: quantifying the additional degree of smoothness ensured by convolving with $\nu$, in terms of the geometry of $\nu$. A particular instance of this problem that has received much attention is that of $L^p$ improvement: what measures $\nu$ have the property that the operator $T_\nu(f)= f*\nu$ maps $L^2$ to $L^{2+\e}$ for some $\e>0$? We recall that if $T_\nu$ maps $L^2$ to $L^{2+\e}$ then by interpolation it maps $L^p$ to $L^{p+\e(p)}$ for all $p\in (1,\infty)$ and some positive $\e(p)$. See \cite{Christ85, Hare88,HareRoginskaya05,Hare15, DooleyHareRoginskaya16} and references therein for progress on this problem. We remark that while wide classes of measures are known to have the $L^p$-improving property, including many Cantor-type measures such as the Cantor-Lebesgue measure on the ternary Cantor set, it remains open whether all Ahlfors-David regular measures of positive exponent on the real line are $L^p$-improving.

In this article we consider a different, but related, notion of smoothness: $L^q$ dimension. Given $\nu\in\PP(\R)$ (the family of compactly supported Borel probability measures on the line), this is defined as
\[
D(\nu,q) = \liminf_{m\to\infty} \frac{ -\log \sum_{Q\in\DD_m} \nu(Q)^q }{ (q-1)m },
\]
where $\DD_{m}$ denotes the collection of level $m$ dyadic half open intervals, that is,
\[
 \DD_m = \left\{ \left[\frac{k}{ 2^{m} },\frac{k+1}{ 2^{m} }\right) : k\in\Z \right\}.
\]
Here, and throughout the paper, the logarithms are to base $2$. The term $(q-1)$ is a normalizing factor that ensures $D(\nu,q)\in [0,1]$.

To emphasize the connection with $L^q$ norms, we note that $D(\nu,q)$ can alternatively be defined as follows: let
\[
\nu_m = 2^m \sum_{Q\in\DD_m} \nu(Q) \mathbf{1}_{Q}.
\]
Then $\nu_m$ is the density of a measure that is an absolutely continuous approximation of $\nu$ at scale $2^{-m}$, and
\begin{equation} \label{eq:Lq-dim-from-Lq-norm}
D(\nu,q) = 1 - \limsup_{m\to\infty} \frac{\log \|\nu_m\|_q^q}{(q-1) m}.
\end{equation}
(Alternatively, we could take $\nu_m$ to be the convolution of $\nu$ with a bump function adapted to the interval $[0,2^{-m}]$.) In particular, if $\nu$ has an $L^q$ density, then $D(\nu,q)=1$ (but not conversely). It is well known and easy to see that $q\mapsto D(\nu,q)$ is non-increasing. See \cite{FanLauRao02} for further background on $L^q$ dimensions and their relationships with other notions of dimension of a measure such as entropy and Hausdorff dimensions.

We also define the $L^\infty$ dimension of $\nu\in\PP(\R)$, denoted by $D(\nu,\infty)$, as the limit of $D(\nu,q)$ as $q\to\infty$ (which exists since $D(\nu,q)$ is decreasing and bounded from below by $0$). It is easy to see that $D(\nu,\infty)$ is the supremum of the Frostman exponents of $\nu$, i.e. of all $s\ge 0$ such that there is $C_s$ satisfying
\[
\nu(B(x,r)) \le C_s r^s \quad\text{for all } x\in\R, r>0.
\]
The values of $q$ are always assumed to be in $(1,\infty)$; when we want to refer to $D(\nu,\infty)$ we will do so explicitly.

We are interested in the growth of $L^q$ dimension under convolutions. Recall that if $\mu,\nu\in\PP(\R)$, then their convolution $\mu\ast\nu\in\PP(\R)$ is the projection of the product measure $\mu\times\nu$ under the addition map, i.e.
\[
 \mu\ast\nu(B) = (\mu\times\nu)(S^{-1}(B)),
\]
where $S(x,y) = x + y$. In keeping with the principle that convolution is a smoothing operation, it always holds that
\begin{equation}
 \label{eq:Lq_improving}
 D(\mu \ast \nu , q) \geq D( \mu , q )
\end{equation}
for any measures $\mu,\nu \in \PP(\R)$. This is a well-known consequence of the convexity of $x\mapsto x^q$, see e.g. \cite[Lemma 2.2]{FengNguyenWang2002}.

The question we are interested in is: when is there strict inequality in \eqref{eq:Lq_improving}? There are two trivial situations that prevent this: if already $D(\mu,q)=1$,  or if $\nu$ is a Dirac mass, then there is equality in \eqref{eq:Lq_improving}. However equality may also hold in situations which are far from these two extreme cases. For example, let $n_j$ be a rapidly increasing sequence ($n_j=j!$ works), and set $E=\cup_j [n_j,2 n_j]$.  Let $\widetilde{\mu}$ be the distribution of an independent sequence of random variables $(\omega_n)_{n=1}^\infty$ such that $\omega_n=0$ for $n\in E$ and $\omega_n$ is either $0$ or $1$ with probability $1/2$ for $n\in\N\setminus E$. Finally, let $\mu$ be the projection of $\widetilde{\mu}$ under the binary expansion map. It is not hard to see that $D(\mu,q)=1/2$ for all $q$ and, indeed, also $D(\mu\ast\mu,q)=1/2$ for all $q$, since $\mu\ast\mu$ has essentially the same structure. Using variants of this example, one can construct measures $\mu,\nu$ of any $L^q$-dimensions (possibly different from each other) such that $D(\mu*\nu,q)=\max\{ D(\mu,q),D(\nu,q) \}$. The key feature of these examples is that, even though \emph{globally} $\mu$ looks nothing like a Dirac mass or Lebesgue measure, \emph{locally and at a finite scale} it looks very much like one or the other, depending on the scale.

Our main result says, informally, that if $\nu$ is uniformly far from being atomic at all places and scales, then $D(\mu*\nu,q)>D(\mu,q)$ unless $D(\mu,q)=1$. To make this precise, we introduce the notion of \emph{uniformly perfect} sets and measures. A set $A\subset\R$ is called \emph{uniformly perfect} if there exists a constant $K > 1$ such that if $x\in A\not\subset B(x, K r)$ for some $r>0$, then $A\cap \big(B(x, K r)\setminus B(x,r)\big)\neq\varnothing$. We extend this definition to measures as follows: we say that a measure $\nu\in\PP(\R)$ \emph{uniformly perfect}, if there exist positive constants $N>1$ (not necessarily an integer) and $0 < \gamma \leq 1$, so that if $\spt\nu\not\subset B(x,N r)$, then
 \begin{equation}
  \label{eq:unif_perf}
  \nu(B(x,N r)) \geq 2^{\gamma}\nu(B(x,r)).
 \end{equation}
Hence, $\nu$ is uniformly perfect if the measure of a ball is not heavily concentrated near its center, in a uniform way. When we want to emphasize the involved constants, we use the term $(N,\gamma)$-uniformly perfect. This condition has also been called reverse doubling.
\begin{theorem}
 \label{thm:unif_perf_improving}
Given $q\in (1,\infty)$, $N > 1$ and $\gamma,\eta > 0$, there is $\eps = \eps(q,N,\gamma,\eta)>0$, such that the following holds:

 If $\nu\in\PP(\R)$ is an $(N,\gamma)$-uniformly perfect measure, then
 \[
 D(\mu\ast\nu,q) >  D(\mu,q)+\eps
 \]
 for any measure $\mu\in\PP(\R)$ with $D(\mu,q) \le 1-\eta$.
\end{theorem}
We remark that in the special case when $\nu$ is a homogeneous self-similar measure on the real line and $q=2$ this statement was proved in \cite[Theorem 4.1]{MosqueraShmerkin2018}, with quantitative estimates. The proof there strongly uses self-similarity and therefore does not extend to more general measures.

Using Theorem \ref{thm:unif_perf_improving} repeatedly, we obtain Corollary \ref{cor:unif_perf_convolution_to_one}, which says that if $\{\mu_i\}_{i\in\N}$ is a family of $(N,\gamma)$-uniformly perfect measures, then $D(\ast_{i=1}^n \mu_i,\infty)$ converges to $1$.

Recall that $\nu\in\PP(\R)$ is called Ahlfors $\alpha$-reqular with constant $C\geq 1$ if
\begin{equation}
\label{eq:ahlfors}
 C^{-1} r^{\alpha} \leq \nu( B(x,r) ) \leq C r^{\alpha} \text{ for all } x\in\spt\nu\text{ and } r\in(0,1).
\end{equation}
It is not hard to see that if $\nu$ is Ahlfors $\alpha$-regular with constant $C>0$, then $\nu$ is uniformly perfect with parameters depending on $C,\alpha$; see Lemma \ref{lem:Ahlfors_is_unif_perf}. Hence we have the following corollary:
\begin{corollary}
\label{cor:DZ}
Given $q\in (1,\infty)$, and $C \geq 1,\ \alpha,\eta\in (0,1)$, there is $\eps = \eps(q,C,\alpha,\eta)>0$, such that the following holds:

 If $\nu\in\PP(\R)$ is  Ahlfors $\alpha$-regular with constant $C$, then
 \[
 D(\mu\ast\nu,q) >  D(\mu,q)+\eps
 \]
 for any measure $\mu$ with $D(\mu,q) < 1-\eta$.
\end{corollary}
This improves upon an earlier result of Dyatlov and Zahl, namely \cite[Theorem 6]{DyatlovZahl2016}, asserting that if $\mu$ is Ahlfors $\alpha$-regular, with $ \alpha \in (0,1)$ and $C\geq1$, then
\[
 D(\mu \ast \mu,2) \geq D(\mu,2) + \eps
\]
where $\eps>0$ depends on $\alpha$ and $C$. We note that $D(\mu,2)=\alpha$ by Ahlfors-regularity, and hence $D(\mu,2)$ is bounded away from $1$.

We deduce Theorem \ref{thm:unif_perf_improving} from an inverse theorem for the $L^q$ norm of convolutions from \cite{Shmerkin2018}, see \S\ref{subsec:inverse-thm} below. The statement of this inverse theorem is fairly technical and one of the motivations of this work is to demonstrate its usefulness through several applications.

We remark that Hochman's inverse theorem for entropy \cite[Theorem 2.7]{Hochman2014} can be used to obtain similar statements for the entropy, and even Hausdorff, dimension of measures. Since $L^q$ dimensions are smaller or equal than both Hausdorff and entropy dimensions, our results cannot be derived from Hochman's theorem in any simple way, and are in many cases stronger. For example, our Corollary \ref{cor:unif_perf_convolution_to_one} immediately implies that if $\{\mu_i\}_{i\in\N}$ is a family of $(N,\gamma)$-uniformly perfect measures, then the Hausdorff dimension of $\ast_{i=1}^n \mu_i$ converges to $1$ at a rate that depends only on $N$ and $\gamma$.

In Section \ref{sec:preliminaries} we introduce notation and recall our main tool, the inverse theorem from \cite{Shmerkin2018}. In Section \ref{sec:proofs} we prove Theorem \ref{thm:unif_perf_improving} as a consequence of a more quantitative result, prove a dual statement for measures with porous supports, and deduce consequences for repeated convolutions. In Section \ref{sec:applications} we present classes of measures to which our results apply. Finally, the relationship with the dimension of sumsets is discussed in Section \ref{sec:sets}.

\section{Preliminaries}
\label{sec:preliminaries}

\subsection{Discretizations}

We introduce some notation. A set $A$ contained in $2^{-m}\Z\cap [0,1)$ is called a \emph{$2^{-m}$-set}. A probability measure is called a  \emph{$2^{-m}$-measure} if it is supported on a $2^{-m}$-set. We will study measures via their discretizations. For $\mu\in\PP(\R)$ we define the level-$m$ discretization of $\mu$ by collapsing the mass of every interval $Q\in\DD_m$ to its left endpoint, that is,
\[
 \mu^{(m)}(k 2^{-m}) = \mu\left( [k 2^{-m},(k+1)2^{-m}) \right)
\]
for all $k\in\Z$. Note that if $\mu\in \PP( [0,1) )$ then $\mu^{(m)}\in\PP( [0,1) )$ is a $2^{-m}$-measure, and $\mu^{(m)}(Q) = \mu(Q)$ for all $Q\in\DD_n$ and $n\leq m$.
Now we can equivalently set the definition of $L^q$ dimension as
\[
 D(\mu,q) = \liminf_{m\to\infty} \frac{ -\log \| \mu^{(m)} \|_q^q }{ (q-1)m },
\]
where for a finitely supported measure $\nu$ we define its $L^q$ norm  $\| \nu \|_q$ as
\[
 \| \nu \|_q^q =  \sum_{x\in\spt\nu} \nu(x)^q .
\]
The minimal value for $\|\mu\|_q$ among all $2^{-m}$-measures is $2^{-m/q'}$, where $q'=q/(q-1)$ is the dual exponent, and this is attained only by the uniform distribution on $2^{m}$ points.

We recall a useful Lemma saying that the order in which discretization and convolution are performed does not affect the $L^q$ norm too much. For the proof, see e.g. \cite[Lemma 4.3]{Shmerkin2018}.
\begin{lemma}
 \label{lem:discret_convolution}
  For any $q\in[1,\infty)$, there is a constant $C_q$, so that
  \[
    C_q^{-1} \|\mu^{(m)}\ast\nu^{(m)}\|^q_q \leq \|(\mu\ast\nu)^{(m)}\|^q_q \leq C_q \|\mu^{(m)}\ast\nu^{(m)}\|^q_q.
  \]
  holds for all $\mu,\nu\in\PP([0,1))$ and $m\in\N$.
\end{lemma}

\subsection{An inverse theorem for the $L^q$ norm of convolutions}
\label{subsec:inverse-thm}

A crucial tool in our analysis is the inverse theorem \cite[Theorem 2.1]{Shmerkin2018}. To state the theorem, we first give some convenient definitions:

Given a set $A\subset\R$, we denote the intervals in $\mathcal{D}_n$ that intersect $A$ by $\mathcal{D}_n(A)$, and denote its cardinality by $N_n(A)=|\mathcal{D}_n(A)|$. For an interval $I$ and $R>0$, we let $R I$ be the interval with the same centre as $I$ and length $R$ times that of $I$.

Let $D,\ell$ be large integers and write $m=D\ell$. A $2^{-m}$-set $A$ is called \emph{$(D,\ell,R_s)$-uniform}, where $(R_s)_{s\in[\ell]}$ is a sequence of integers in $[1,2^D]$, if $N_{(s+1)D}( A\cap J) = R_s$ for all $J\in \mathcal{D}_{sD}( A )$. Here, and throughout the paper $[\ell]$ stands for the set of integers $\{0,1,\ldots,\ell-1\}$. One can picture this definition using trees: if one considers the tree whose vertices of level $s$ are the intervals in $\mathcal{D}_{s D}(A)$ (with descendence given by inclusion), then $A$ is $(D,\ell,R_s)$-uniform if in the associated tree each vertex of level $s$ has exactly $R_s$ descendants.  For this reason we sometimes refer to the values $R_s$ as the  \emph{branching numbers} of $A$. The set $A$ is said to be \emph{uniform} if there are numbers $D,\ell$ and $(R_s)_{s\in[\ell]}$ so that $A$ is $(D,\ell,R_s)$-uniform. We emphasize that uniformity does not say anything about the distribution of the subintervals $I$ inside $J$, only about their cardinality.

\begin{theorem}
\label{thm:inverse}
For each $q>1$, $\delta>0$, and $D_0\in\N$, there are $D \geq D_0$ and $\eps > 0$, so that the following holds for $\ell\ge \ell_0(q,\delta,D_0)$.

Let $m=\ell D$ and let $\mu$ and $\nu$ be $2^{-m}$-measures with
\[
 \|\mu \ast \nu\|_q \geq 2^{-\eps m} \|\mu\|_q.
\]
Then there exist $2^{-m}$-sets $A\subset \spt\mu$ and $B\subset\spt\nu$, numbers $k_A, k_B\in 2^{-m}\Z$, and a set $\mathcal{S}\subset[\ell]$, so that
\begin{labeledlist}{A}
 \item\label{A1} $\| \mu|_A \|_q \geq 2^{-\delta m}\|\mu\|_q$.
 \item\label{A2} $\mu(x) \leq 2 \mu(y)$ for all $x,y\in A$.
 \item\label{A3} $A'=A+k_A$ is contained in $[0,1)$ and is $(D,\ell,R')$ uniform for some sequence $R'$.
 \item\label{A4} $x\in \frac{1}{2}\mathcal{D}_{sD}(x)$ for each $x\in A'$ and $s\in [\ell]$.
\end{labeledlist}
\begin{labeledlist}{B}
 \item\label{B1} $\| \nu|_B \|_1 = \nu(B) \geq 2^{-\delta m}$.
 \item\label{B2} $\nu^{(m)}(x) \leq 2 \nu(y)$ for all $x,y\in B$.
 \item\label{B3} $B'=B+k_B$ is contained in $[0,1)$ and is $(D,\ell,R'')$ uniform for some sequence $R''$.
 \item\label{B4} $y\in \frac{1}{2}\mathcal{D}_{sD}(y)$ for each $y\in B'$ and $s\in [\ell]$.
\end{labeledlist}
Moreover
\begin{labeledlist}{}
\setcounter{enumi}{4}
 \item\label{inverse5} for each $s$, $R''_s=1$ if $s\not\in\mathcal{S}$, and $R'_s \geq 2^{(1-\delta)D}$ if $s\in\mathcal{S}$.
 \item\label{inverse6} The set $\mathcal{S}$ satisfies
 \[
  \log \|\nu\|^{-q'}_q - m\delta \leq D|\mathcal{S}| \leq \log \|\mu\|^{-q'}_q + m\delta.
 \]
\end{labeledlist}
\end{theorem}

\begin{remark}
In \cite[Theorem 2.1]{Shmerkin2018}, both the convolution and the translation take place on the circle $[0,1)$ with addition modulo $1$. The translation is given by \cite[Lemma 3.8]{Shmerkin2018}, and it follows from the construction in this lemma that $A+x_A\subset [0,1)$ as subsets of the real line. Note that the convolution $\mu*\nu$ is in general different if taken on the circle rather than the real line, but each can be obtained from the other by dividing the resulting measure into two pieces and translating one of them. From this fact it is easy to see that the $L^q$ norms of both convolutions differ by at most a multiplicative constant (this follows e.g. from \cite[Lemma 4.2]{Shmerkin2018}).
\end{remark}

\section{Improvement of $L^q$ dimension under convolution}
\label{sec:proofs}

\subsection{Proof of Theorem \ref{thm:unif_perf_improving}}

In this section we obtain Theorem \ref{thm:unif_perf_improving}. We begin with some observations about uniformly perfect measures which will be used throughout the paper. It is equivalent to assume that \eqref{eq:unif_perf} holds only for $x\in\spt\nu$ (at the price of changing the value of $N$). Indeed, \eqref{eq:unif_perf} holds trivially if  $\nu(B(x,r))=0$; otherwise, pick $y\in B(x,r)\cap \spt\nu$ and note that
\[
\nu(B(x,(2N+1)r)) \ge \nu(B(y,2N r)) \ge 2^\gamma \nu(B(y, 2r)) \ge 2^\gamma\nu(B(x,r)).
\]
In particular, the support of a uniformly perfect measure is a uniformly perfect set. For more on the relationship between uniformly perfect sets and measures, see \S\ref{sec:sets}.

If $\nu$ is a $2^{-m}$-measure, then we say that $\nu$ is $(N,\gamma)$-uniformly perfect if \eqref{eq:unif_perf} holds for all $r\ge 2^{-m}$ (for this condition to be meaningful we need $2^m\gg N$). It will be important for us later that if $\nu$ is $(N,\gamma)$-uniformly perfect, then $\nu^{(m)}$ is $(2N+1,\gamma)$-uniformly perfect which holds since (for $r\ge 2^{-m}$)
\[
\nu^{(m)}(B(x, (2N+1)r)) \ge \nu(B(x,2N r)) \ge 2^\gamma \nu(B(x,2r)) \ge 2^\gamma \nu^{(m)}(B(x,r)).
\]

Theorem \ref{thm:unif_perf_improving} will be obtained as a corollary of the following more quantitative result:
\begin{proposition}
 \label{prop:lq_norm_unif_perf}
Given $q > 1$, $N > 1$ and $\gamma,\eta,a > 0$, there is $\eps = \eps(q,N,\gamma,\eta)>0$, such that the following holds for all large enough $m=m(q,N,\gamma,\eta,a)$:

Let $\nu\in\PP(\R)$ be a $(2^N,\gamma)$-uniformly perfect $2^{-m}$-measure whose support has diameter $\ge a$. Then
 \[
  \|\mu\ast\nu\|_q < 2^{-\eps m} \|\mu\|_q,
 \]
 for all $2^m$-measures $\mu$ with $\|\mu\|_q > 2^{-m(1-\eta)/q'}$.
\end{proposition}
\begin{proof}

Let $\delta= \eta\gamma / (5N)$ (so that, in particular, $\delta < \eta/2$), and $D_0 = \lceil 2N+2 \rceil$. Let $\eps>0$ and $D\in\N$ be given by Theorem \ref{thm:inverse} for $\delta$ and $D_0$, and let $\ell_1=\ell_1(\delta,D_0)$ be such that the conclusions of the theorem hold for $\ell\ge\ell_1$.

We assume first that $m=D\ell$ for some $\ell\ge \ell_1$ and comment at the end how to extend the claim to general $m$. Fix, then $m=D\ell\ge D\ell_1$ and let $\mu$ and $\nu$ be as in the statement.  Assume, for the sake of contradiction, that
\begin{equation}
 \label{eq:contra_1}
 \|\mu \ast \nu\|_q \geq 2^{-\eps m} \|\mu\|_q.
\end{equation}
Applying  Theorem \ref{thm:inverse}, we get uniform $2^{-m}$-sets $A\subset\spt\mu$ and $B\subset\spt\nu$, and a set $\SSS\subset[\ell]$ corresponding to the scales of almost full branching of $A$. By replacing $\mu$ and $\nu$ by their translations by $k_A, k_B$ respectively, which does not change any of the claims we make about them, we may and do assume that $k_A=k_B=0$.

Since $\|\mu\|_q \geq 2^{-m(1-\eta)/q'}$, part \ref{inverse6} of Theorem \ref{thm:inverse} gives
\begin{equation}
 \label{eq:S_bound}
  D|\SSS|
 \leq
 \log ( \|\mu\|_q^{-q'} ) + \delta m
  \leq
 (1-\eta+\delta) m
  \leq
 (1-\eta/2) D\ell.
\end{equation}
Recall that $D \geq D_0 \geq N$. Consider $s\in[\ell]$, and intervals $J\in\DD_{sD}$ and $I\in\DD_{(s+1)D}(\spt\nu)$ such that $I\subset\frac{1}{2}J$ and $\spt\nu\not\subset J$. Set $I^0=I$, and let $I^1 = 2^{N}I^0$. By the uniform perfectness of $\nu$ , we have that $\nu(I^0) \leq 2^{-\gamma} \nu( I^1 )$.  We continue inductively until we reach the largest $K\in\N$ for which
\[
 2^{NK} 2^{-(s+1)D} = \diam I^{K} \leq \frac{1}{2} \diam J = 2^{-sD-1}.
\]
That is, $K = \big\lfloor \frac{D-1}{N} \big\rfloor$. Since $2^{NK} I\subset J$, we have now gotten that
\begin{equation}
 \label{eq:measure_cut}
 \nu(I) \leq 2^{-\gamma \lfloor \frac{D-1}{N} \rfloor} \nu(J)
\end{equation}
for all $J\in\DD_{sD}$ and $I\in\DD_{(s+1)D}$ such that $I\in\frac{1}{2}J$ and $\spt\nu\not\subset J$.

Let $\ell_0$ be the smallest integer such that $2^{-\ell_0}<a\le \diam(\spt\nu)$. By making $\ell_1$ larger in terms of $a$ and $\eta$, we may assume that $\ell_1>10\ell_0/\eta$. Let $L\subset \{\ell_0,\ell_0+1,\ldots,\ell\}$ denote the scales where $B$ has no branching, i.e. the set of $s\in \{\ell_0,\ell_0+1,\ldots,\ell\}$ for which $R''_s=1$. By part \ref{inverse5} of Theorem \ref{thm:inverse}, \eqref{eq:S_bound} and our choice $\delta=\eta\gamma / (5N)$, we have that
\[
 |L|\geq \ell-|\mathcal{S}|-\ell_0 \geq \frac{\eta}{2}\ell -\ell_0 > 2 \delta \gamma^{-1} N \ell.
\]
Part \ref{B1} of the inverse theorem gives that $2^{-\delta m}\leq\nu(B)$. On the other hand \ref{B4} and our choice of $\ell_0$ allow us to use \eqref{eq:measure_cut} for each $I\in\DD_{(s+1)D}(B)$ with $s\in L$, so we get
\begin{align*}
 2^{-\delta m}
 &\leq
 \nu(B)
 \leq
 2^{-\gamma \lfloor \frac{D-1}{N} \rfloor|L|}.
\end{align*}
The last two displayed equations yield
\[
D\ge 2\left\lfloor \frac{D-1}{N}\right\rfloor N > 2(D-1) - 2N
\]
which, since we assumed $D\ge D_0 \ge 2N+2$, is a contradiction. This finishes the proof, under the assumption that $m=D\ell$.

In the general case, take $m=D\ell +j$ with $j\in\{0,1,\ldots,D-1\}$ and $\ell\ge \ell_1$. If $\rho$ is a $2^{-m}$-measure then, decomposing
\[
\|\rho\|_q^q = \sum_{I\in\DD_{D\ell}} \rho(I)^q \sum_{J\in\DD_m, J\subset I} (\rho(J)/\rho(I))^q
\]
it follows that
\[
 \|\rho\|_q \le   \|\rho^{(D\ell)} \|_q \le 2^{(D-1)/q'}\|\rho\|_q .
\]
On the other hand, we have seen that if $\nu$ is $(N,\gamma)$-uniformly perfect, then $\nu^{(D\ell)}$ is $(2N+1,\gamma)$-uniformly perfect. Combining these facts with Lemma \ref{lem:discret_convolution} and the already established case $m=D\ell$, we conclude that
\begin{align*}
\| \mu*\nu\|_q &\le \| (\mu*\nu)^{(D\ell)}\|_q \le C_q \| \mu^{(D\ell)} * \nu^{(D\ell)}\|_q \\
&\le C_q 2^{-\e D\ell} \|\mu^{(D\ell)}\|_q \le C_q  2^{-\e D\ell} 2^{(D-1)/q'} \|\mu\|_q \\
&\le 2^{-\e m/2} \|\mu\|_q
\end{align*}
if $m$ is taken large enough in terms of $D$ and $q$, and hence in terms of the given parameters.
\end{proof}

\begin{proof}[Proof of Theorem \ref{thm:unif_perf_improving}]
Since scaling $\mu$ and $\nu$ by the same factor and translating them does not affect the statement, we may assume that $\mu$ and $\nu$ are both supported on $[0,1)$. Since $D(\mu,q) \le 1-\eta$, we must have
\[
\|\mu^{(m)}\|_q > 2^{-m(q-1)q^{-1}(1-\eta/2)} = 2^{-m(1-\eta/2)/q'}
\]
for infinitely many $m$. Since $\nu^{(m)}$ is $(2N+1,\gamma)$-uniformly perfect, Proposition \ref{prop:lq_norm_unif_perf} gives $\eps$ and $D$ (depending on $q,N,\gamma,\eta$) so that
\begin{equation}
 \label{eq:ahlfors_lqnorm}
 \| \nu^{(m)} \ast \mu^{(m)} \|_q < 2^{-\eps m} \| \mu^{(m)} \|_q
\end{equation}
for infinitely many $m=D\ell$. For the remaining large $m$, we use the estimate
\[
\| \mu^{(m)} \ast \nu^{(m)} \|_q \leq \|\mu^{(m)}\|_q \leq 2^{-m(q-1)q^{-1}(1-\eta/2)},
\]
where the first inequality follows from the convexity of $t\mapsto t^q$. Combining these estimates, and using Lemma \ref{lem:discret_convolution} we get
\[
 \| (\nu\ast\mu)^{(m)} \|_q^q
 \leq
 C_q \| \nu^{(m)} \ast \mu^{(m)} \|_q^q
 \leq
 C_q\min\left\{ 2^{-\eps q m} \| \mu^{(m)} \|_q^q \ ,\ 2^{-m(q-1)(1-\eta/2)} \right\}
\]
for all $m$ large enough. Thus
\begin{align*}
 \frac{ -\log \| (\nu\ast\mu)^{(m)} \|_q^q }{ (q-1)m }
 &\geq
 \min\left\{
 \frac{ -\log \left(C_q 2^{-\eps q m} \| \mu^{(m)} \|_q^q \right)}{ (q-1)m }\ ,\ \frac{ -\log\left( C_q 2^{-m(q-1)(1-\eta/2)}\right)  }{ (q-1)m }
 \right\}
\end{align*}
for all $m$ large enough. Taking liminf of both sides in the above inequality yields
\[
 D(\mu\ast\nu,q) \geq D(\mu,q) + \min\left\{ \eps\frac{ q }{ q-1 } \ ,\ \eta/2 \right\}
\]
ans so the proof is finished.
\end{proof}

\subsection{Porosity and the $L^q$ norm of convolutions}

Theorem \ref{thm:unif_perf_improving} says that if $\mu$ is uniformly perfect then $D(\mu*\nu,q)>D(\nu,q)$ unless $D(\nu,q)$ is already maximal. In
Theorem \ref{thm:revese_improving_porous} below we consider the dual problem of giving conditions on $\mu$ such that $D(\mu*\nu,q)>D(\mu,q)$. We will see that this holds whenever $\spt\mu$ is porous and $D(\nu,p)$ is positive for any $p$.

We now recall the notion of (lower) porosity. A set $A\subset\R$ is called $\alpha$-porous, if there is $\alpha>0$ so that for every $r>0$ and every $x\in \R$, there is $y\in\R$ so that $B(y,\alpha r) \subset B(x,r)$, but $B(y,\alpha r) \cap A = \varnothing$. There are many different notions of porosity, and what we are using here is the strongest one. Sometimes the condition is required to hold only  for all $r$ smaller than some $r_0$ and $x\in A$, but this would only affect the constant $\alpha$, and since we do not care too much about the values of our constants, we stick with the simpler definition. Further, it is convenient for us to work with the following dyadic version: we say that a set $A\subset\R$ is dyadic $k$-porous, or just $k$-porous for short, if for all $n$ and $J\in\DD_n$ there is $I\in\DD_{n+k}(J)$ with $I\cap A = \varnothing$. Dyadic porosity is easily seen to be equivalent with  Euclidean porosity, in the sense that for all $\alpha > 0$ there is $k=k(\alpha)$ so that any $\alpha$-porous set is dyadic $k$-porous  and conversely.

The definition of $\alpha$-porosity makes sense for discrete sets, but similarly as for uniform perfectness, it is useful only when $2^{-m} \ll \alpha$. This is not an issue, since we can always assume $m$ to be as large as we wish. Again, the point to keep in mind is that if $\spt\mu$ is $\alpha$-porous, then $\spt \mu^{(m)}$ is $\alpha/2$-porous for all $m$ large enough.

\begin{proposition}
 \label{prop:revese_discrete_improving_porous}
 For given $p,q>1$ and $\alpha,\sigma>0$, there is $\eps = \eps(p,q,\alpha,\sigma)>0$, so that the following holds for all large enough $m = m(p,q,\alpha,\sigma)$:

 Let $\mu$ be a $2^{-m}$-measure such that $\spt\mu$ is $\alpha$-porous. Then $$\| \mu \ast \nu \|_q < 2^{-\eps m} \| \mu \|_q$$ for all $2^{-m}$-measures $\nu$ with $\|\nu\|_p < 2^{-\sigma m/p'}$.
\end{proposition}
\begin{proof}
As in the proof of Proposition \ref{prop:lq_norm_unif_perf}, we may restrict to the case $m=D\ell$, for some fixed $D=D(p,q,\alpha,\sigma)$.

 Let $k=k(\alpha)$ so that $\alpha$-porous sets are dyadic $k$-porous, and let $\eta = \log(2^k - 1)/\log 2^k$,  so that $(2^k - 1) = (2^k)^\eta$. Choose $\delta$ so that $ \sigma  / p' > \delta > 0$ and $1-\delta > \eta+\delta$, and then pick $D_0$ so that $k/D_0 < \delta $. Let $\eps>0$ and $D\ge D_0$ be given by Theorem \ref{thm:inverse} for $q,\delta$ and $D_0$.

 Assume on the contrary that $\| \mu \ast \nu \|_q \geq 2^{-\eps m} \| \mu \|_q$ and apply Theorem \ref{thm:inverse} to the measures $\mu$ and $\nu$ to obtain the uniform $2^{-m}$-sets $A\subset \spt\mu$ and $B\subset \spt\nu$.

 From the porosity of $\spt\mu$ we get an estimate for the branching numbers $R'_s$ of $A'$. Indeed, since for every $I\in\DD_s(A')$ there is $J\in\DD_{s+k}(I)$ such that $A'\cap J=\varnothing$, we get
 \[
 R'_s
 \leq (2^k - 1)^{\lfloor \frac{D}{k} \rfloor} 2^k
 \leq (2^{k\eta})^{\frac{D}{k}} 2^{ k}
 \leq 2^{(\eta+k/D)D }
 < 2^{(1-\delta)D}\ ,
 \]
 using the choice of $\delta, D_0$ in the last inequality. According to part \ref{inverse5} of the inverse theorem, we now have that $\SSS$ is actually the empty set, and thus $B$ is a singleton.

 By \ref{B1}, we have $\nu(B) \geq 2^{-\delta m}$, which now trivially gives $\|\nu\|_p \geq 2^{-\delta m}$. On the other hand, we assumed that $ \|\nu\|_p < 2^{-\sigma m /p'}$. Combining these estimates gives $\delta > \sigma /p'$, which contradicts the choice of $\delta$.
\end{proof}
\begin{remark}
 \label{rem:on_discrete_assumptions}
 In Proposition \ref{prop:lq_norm_unif_perf} the assumption $\|\mu\|_q > 2^{-m(1-\eta)/q'}$ is simply saying that the $L^q$ norm of $\mu$ is not too close to minimal to begin with. Similarly, in Proposition \ref{prop:revese_discrete_improving_porous}, $\|\nu\|_p < 2^{-\sigma m/p'}$ only says that $\nu$ does not have almost maximal $L^p$ norm.
\end{remark}

\begin{theorem}
 \label{thm:revese_improving_porous}
 For given $p,q>1$ and $\alpha,\sigma>0$, there exists $\eps = \eps(p,q,\alpha,\sigma)>0$ so that the following holds:

 Let $\spt\mu$ be $\alpha$-porous. Then $D(\mu \ast \nu,q) \geq  D(\mu,q)+\eps$ for any $\nu\in\PP(\R)$ with $D(\nu,p) > \sigma$.
\end{theorem}
\begin{proof}
As in the proof of Theorem \ref{thm:unif_perf_improving}, we may assume that $\mu$ and $\nu$ are supported on $[0,1)$. We have that $\mu^{(m)}$ is $\alpha/2$-porous for all $m$ large enough, and $D(\nu,p) > \sigma$ implies that $\| \nu^{(m)} \|_p^p < 2^{-(p-1) \sigma m}$ for all $m$ large enough. Thus the assumptions of Proposition \ref{prop:revese_discrete_improving_porous} are satisfied for all large $m$. The rest of the proof goes like the proof of Theorem \ref{thm:unif_perf_improving} except that here there is no need to deal with different cases since the same estimate holds for all $m$.
\end{proof}

\subsection{Repeated convolutions}

Next, we derive a corollary for repeated convolutions. First, we state a useful inequality; it follows from Young's convolution inequality, and a detailed argument can be found in \cite[Lemma 5.2]{MosqueraShmerkin2018}.
\begin{lemma} \label{lem:convolution-young}
For any $\mu,\nu\in \PP(\R)$ it holds that
\[
D(\mu \ast\nu,\infty) \geq (D(\mu,2) + D(\nu,2)) / 2
\]
\end{lemma}

\begin{corollary}
 \label{cor:unif_perf_convolution_to_one}
 Let $\{\nu_i\}_{i\in\N} \subset\PP(\R)$ be a family of measures so that each $\nu_i$ is $(N_i,\gamma_i)$-uniformly perfect. Then, for all $q>1$, the sequence $\left( D(\ast_{i=1}^n \nu_i , q)\right)_n$ is strictly increasing, until the possible point where it reaches $1$.

 Moreover, if $N_i\equiv N$ and $\gamma_i \equiv \gamma$, then the sequence $D( \ast_{i=1}^n \nu_i , \infty)$ converges to $1$ as $n\to \infty$ at a rate that depends only on $N$ and $\gamma$  (and hence so does $D( \ast_{i=1}^n \nu_i ,q)$ for all finite $q$).
\end{corollary}
\begin{proof}
For the first part, suppose $D(\ast_{i=1}^n \nu_i , q ) = 1-\eta$ with $\eta>0$. Let $\eps_{n+1}=\eps_{n+1}(q,N_{n+1},\gamma_{n+1},\eta)$ be the number given by Theorem \ref{thm:unif_perf_improving}. Then choosing $\mu = \ast_{i=1}^n \nu_i$ and $\nu = \nu_{n+1}$ we get $D(\ast_{i=1}^{n+1} \nu_i , q ) \geq  D(\ast_{i=1}^n \nu_i , q )+\eps_{n+1}$.

 If  $N_i\equiv N$ and $\gamma_i\equiv \gamma$, then the same argument shows that for all $\eta>0$, if $\eps=\eps(2,N,\gamma,\eta)$ is the number given by Theorem \ref{thm:unif_perf_improving} applied to $q=2, N,\gamma$ and $\eta$, then
\[
D(\ast_{i=1}^{n} \nu_i , 2) \ge \min(D(\nu_1,2)+(n-1)\eps,1-\eta) = 1-\eta,
\]
if $n=n(N,\gamma,\eta)=1+\lceil (1-\eta)/\eps\rceil$, and the same holds for $\ast_{i=n+1}^{2n} \nu_i$. Applying Lemma \ref{lem:convolution-young} to $\ast_{i=1}^{n} \nu_i$ and $\ast_{i=n+1}^{2n} \nu_i$ finishes the proof.
\end{proof}

\begin{remark}
 \label{rem:not_for_poro}
We cannot apply the same argument to a sequence of measures with porous support, even with positive $L^p$ dimension, since the sum of porous sets can fail to be porous, and so Theorem \ref{thm:revese_improving_porous} is not applicable.
\end{remark}

\section{Classes of uniformly perfect measures}

\label{sec:applications}

\subsection{Ahlfors regular and doubling measures}

In this section we show that some important classes of measures are uniformly perfect. Recall that a set is called  Ahlfors $\alpha$-regular if it is a support of an Ahlfors $\alpha$-regular measure. Ahlfors regular measures are uniformly perfect, and if the exponent is $<1$ then their supports are always porous:
\begin{lemma}
\label{lem:Ahlfors_is_unif_perf}
 If $\nu\in\PP(\R)$ is Ahlfors $\alpha$-regular with constant $C$, then $\nu$ is $( 2(2^\gamma C^2)^{\frac{1}{\alpha}} +1, \gamma )$-uniformly perfect for any $0<\gamma\leq 1$.
\end{lemma}
\begin{proof}
 Let $N^\alpha = 2^\gamma C^2$, and fix $x\in\spt\nu$ and $r>0$ so that $\spt\nu\not\subset B(x, N r)$. By the Ahlfors regularity, we have
 \begin{align*}
  2^{^\gamma}\nu(B(x,r))
  &\leq
  2^{\gamma} C r^{\alpha}
  \leq
  C^{-1} N^\alpha r^{\alpha}
  \leq
  \nu(B(x, N r))
 \end{align*}
 To finish the proof we recall that if a measure is $(N,\gamma)$-uniformly perfect on points of its support, then it is $(2N+1,\gamma$)-uniformly perfect.
\end{proof}

\begin{lemma}
 \label{lem:Ahlfors_is_porous}
 If $\nu\in\PP(\R)$ is Ahlfors $\alpha$-regular, with constant $C\geq 1$ and $\alpha \in (0,1)$, then $\spt\nu$ is dyadic $k$-porous for  $k=\lceil (3 + 2\log C) (1-\alpha)^{-1} \rceil$.
\end{lemma}
\begin{proof}
Let $J\in\DD_{n}(\spt\nu)$ and fix any $x\in\spt\nu \cap J$. Then by definition we have
\begin{equation}
 \label{eq:n_Ahlfors_upper}
 \nu(J) \leq \nu(B(x,2^{-n})) \leq C 2^{-\alpha n}.
\end{equation}
Consider $N:=N_{n+k}(\spt\nu\cap J)$. We can assume that $N \geq 4$. It is now possible to choose a $2\cdot 2^{-(n+k)}$ separated collection $\{y_j\}_{j=1}^{N/4} \subset\spt\nu\cap J$ which is at distance $> 2^{-(n+k)}$ to the boundary of $J$. Thus
\begin{equation}
 \label{eq:branching_upper}
  C 2^{-\alpha n} \geq \nu(J) \geq \sum_{j=1}^{N/4} \nu( B(y_j,2^{-(n+k)}) ) \geq \frac{N}{4} C^{-1} 2^{-\alpha(n+k)},
\end{equation}
giving $N \leq 4 C^2 2^{\alpha k}$. By the choice of $k$, we have that $4 C^2 2^{\alpha k} < 2^k$, and hence $\spt\nu$ is $k$-porous.
\end{proof}

\begin{corollary}
\label{cor:lq_dim_ahlfors}
Given $q>1$, $\eta>0$, $\alpha\in(0,1)$ and $C\geq1$ there is $\e>0$ such that the following holds: if $\mu\in\PP(\R)$ is Ahlfors $\alpha$-regular with constant $C$ and if $\eta<D(\nu,q)<1-\eta$, then
\[
D(\mu\ast\nu,q) > \max(D(\mu,q),D(\nu,q)) + \eps.
\]
\end{corollary}


The following example demonstrates how the case $q=\infty$ is different from the case of finite $q$.
\begin{example}
 \label{ex:the_third_time}
 Let $\mu\in\PP(\R)$ be symmetric around the origin and Ahlfors $\alpha$-regular, with constant $C$ (for example, one can take Hausdorff measure on a central self-similar Cantor set). By the Frostman exponent interpretation, it is clear that $D(\mu,\infty)=\alpha$. For $r > 0$ small, choose a maximal $r$-separated set $\{y_i\}_i\subset\spt\mu$. It is easy to see that it has cardinality at least $ C^{-1} r^{-\alpha}$. By the symmetry assumption, also $-y_i\in\spt\mu$ for all $i$. Thus
 \[
  (\mu\ast\mu)(B(0,r))
  \geq
  C^{-1} r^{-\alpha} ( C^{-1} (r/2)^{\alpha} )^2
  =
  C^{-3} 2^{-2\alpha} r^{\alpha},
 \]
 giving that any Frostman exponent of $\mu\ast\mu$ is at most $\alpha$. This means that $D(\mu\ast\mu,\infty) \leq \alpha$ and thus  $D(\mu*\mu,\infty)=D(\mu,\infty)$. This shows that neither Theorem \ref{thm:unif_perf_improving} nor \ref{thm:revese_improving_porous} can be extended to $q=\infty$.

 We note that, however, in the Ahlfors-regular (or just uniformly perfect) case one does have $D(\mu*\mu*\mu,\infty)>D(\mu,\infty)$. Indeed, by Theorem \ref{thm:unif_perf_improving} applied with $q=2$ and Lemma \ref{lem:convolution-young}, we get
 \begin{align*}
  D(\mu\ast\mu\ast\mu, \infty)
  &\geq
  ( D(\mu\ast\mu , 2) + D(\mu, 2) ) / 2 \\
  &\geq
  ( D(\mu, 2) + \eps + D(\mu, 2) ) / 2 \\
  &\geq
  D(\mu,\infty) + \eps/2.
 \end{align*}

\end{example}

 Recall that a measure $\mu$ is said to be doubling with constant $C$, if for all $x\in\spt\mu$ and $r>0$ it holds that $C \mu(B(x,r)) \geq \mu(B(x,2r))$. We emphasize that we only require the doubling condition for points in the support of the measure. Ahlfors regular measures are doubling, but the converse does not need to hold. We have the following lemma.

\begin{lemma}
\label{lem:doubling_to_unif_perf}
 If $\nu\in\PP(\R)$ has uniformly perfect support and is doubling (on the support), then $\nu$ is uniformly perfect, quantitatively.
\end{lemma}
\begin{proof}
Suppose $\spt\nu$ is $K$-uniformly perfect and let $N=2K+1$. Assume $\spt\nu \not\subset B(x,Nr)$ . By the uniform perfectness of $\spt\nu$, there exists $y\in \spt\nu \cap (B(x,2Kr)\setminus B(x,2r))$. In particular, $B(y,r)\subset B(x,Nr)$ and $B(y,r) \cap B(x,r) =\varnothing$. Let $M=\lceil \log (N+1)\rceil$. By the doubling condition,
\[
\nu(B(y,r)) \ge C^{-M} \nu(B(y,(N+1)r)) \ge C^{-M} \nu(B(x,r)).
\]
Hence $\nu(B(x,Nr)) \ge (1+C^{-M})\nu(B(x,r))$, yielding uniform perfectness.
\end{proof}

\begin{corollary}
\label{cor:lq_dim_doubling}
For given $q\in (1,\infty)$, $K,C\in\N_{\geq 1}$ and $\eta > 0$, there is $\eps = \eps(q,K,C,\eta)>0$, such that the following holds:

If $\nu\in\PP(\R)$ doubling, with constant $C$ and $\spt\mu$ is uniformly perfect with constant $K$, then $D(\mu\ast\nu,q) > D(\mu,q)+\eps$ for any measure $\mu$ with $D(\mu,q) < 1-\eta$.
\end{corollary}

\subsection{Moran constructions}

We obtain a general class of measures that are uniformly perfect via Moran constructions. This class turns out to include, for example, all non-trivial self-similar and self-conformal measures on the line.

Let $\Sigma=\{1,\ldots,\kappa\}$ be a finite alphabet. Denote the set of words of length $n$ by $\Gamma_n = \Sigma^n$ and the set of all finite words by $\Gamma_* = \bigcup_{n\in\N}\Gamma_n$, where we interpret $\Gamma_0=\{\varnothing\}$, with $\varnothing$ being empty word. Set $\Gamma = \Sigma^\N$, the set of right-infinite words. If $\ii=(i_1,i_2,\ldots)\in\Gamma$, then we denote $\iin{n}=(i_1,i_2,\ldots,i_n)$, $[\iin{n}]=\{(j_1,j_2,\ldots)\in\Gamma : j_m=i_m \text{ for all } m\leq n \}$, and $\iin{n}^{-} = (i_1,i_2,\ldots,i_{n-1})\in\Gamma_{n-1}$. We equip $\Sigma$ with the discrete topology, and $\Gamma$ with the induced product topology.

Assume that we have a collection $\{ E_\iii : \iii \in \Gamma_* \}$ of compact sets of positive diameter in $\R$. Such a collection is called a \emph{Moran construction}, if the following conditions are satisfied:
\begin{labeledlist}{M}
  \item $E_\iii \subset E_{\iii^-}$ for all $\iii \in \Gamma_* \setminus \{ \varnothing \}$, \label{M1}
  \item $\diam(E_{\iii|_n}) \to 0$ as $n \to \infty$ for all $\iii \in \Gamma$, \label{M2}
\end{labeledlist}
Given a Moran construction, we define the \emph{projection mapping} $\pi \colon \Gamma \to \R$ by
\begin{equation*}
  \{ \pi(\iii) \} = \bigcap_{n=1}^\infty E_{\iii|_n}
\end{equation*}
for all $\iii \in \Gamma$. The assumptions \ref{M1}, \ref{M2} guarantee that $\pi$ is a well-defined continuous mapping. The compact set $\pi(\Gamma)$ is called the \emph{limit set} of the Moran construction and throughout the paper, we shall denote it by $E$.

Suppose that we have absolute constants $0<p_*<p^*<1$ and that for each $\ii\in\Gamma_*$ we have numbers $p_\ii$ satisfying
\begin{equation} \label{eq:conditions-measure}
p_* \leq p_\ii \leq p^* \quad\text{and}\quad  \sum_{ \ii\in\Gamma_* \text{ and } \ii^-=\jj } p_\ii = 1
\end{equation}
for any $\jj\in\Gamma$. Let $\tilde\nu$ be the unique measure on $\Gamma$ satisfying $\tilde\nu[\ii] = \prod_{n=1}^{|\ii|} p_{\iin{n}}$. We denote the push forward measure $\pi\tilde\nu$ simply by $\nu$ and call it a Moran construction measure. If the factors $p_\ii$ only depend on the last coordinate of $\ii$, then $\tilde\nu$ is called a Bernoulli measure.

In addition to the conditions \ref{M1} and \ref{M2} we need to impose some other geometric conditions, to avoid degenerated situations. Thus we introduce the following conditions
\begin{labeledlist}{M}
\setcounter{enumi}{2}
  \item there exists $\beta \ge 1$ such that $\diam(E_{\iii\jjj}) \le \beta\diam(E_\iii)\diam(E_\jjj)$ for all $\iii\jjj\in \Gamma_*$, \label{M3}
  \item there exists $0<\lalpha<1$ such that $\diam(E_\iii) \ge \lalpha\diam(E_{\iii^-})$ for all $\iii \in \Gamma_* \setminus \{ \varnothing \}$, \label{M4}
  \item there exists $\roo>0$, so that for all $\ii\in\Gamma_*$ there exist $x,y\in E\cap E_\ii$ with $x \neq y$ and $d(x,y)\geq \roo \diam(E_\ii)$. \label{M5}
\end{labeledlist}
The only part that deals with the positions of the construction sets is \ref{M5}. This can be viewed as a very minimal separation condition. Essentially, this requires that each $E_\ii$ contains two pieces, say $E_{\ii j}$ and $E_{\ii k}$, of the construction that are really different, but they may still overlap a lot. If \ref{M5} is satisfied, then by setting $E'_\ii = E_\ii\cap E$ we get a Moran construction, with the same limit set, satisfying \ref{M3}-\ref{M5} with constants $\beta'=\beta \roo^{-2}$, $\lalpha'=\lalpha\roo$ and $\roo'=1$. With this at hand, from now on we implicitly assume that $E_\ii=E\cap E_\ii$, and so $\roo = 1$.


\begin{proposition}
 If $\{E_\ii:\ii\in\Gamma_*\}$ is a Moran construction satisfying \ref{M1}--\ref{M5} and $\nu\in\PP(\R)$ is an associated Moran construction measure, then $\nu$ is uniformly perfect.
\end{proposition}
\begin{proof}
By rescaling and translating, we may assume that $\{0,1\}\subset E\subset[0,1]$. Let $M$ be the smallest integer with $\diam E_\kk \leq \frac{\alpha}{4} \beta^{-1} $ for all $\kk\in\Gamma_M$, which exists by \ref{M2} and compactness. 

Let $I$ be an interval intersecting $E$, but with $E \not\subset 12\lalpha^{-1} I$. Consider the collection
\[
 \Phi(I) = \{ \ii\in\Gamma : E_{\ii}\cap I \neq\varnothing\text{ and } \diam E_{\ii} < 4 \lalpha^{-1} \diam I \leq \diam E_{\ii^-} \}
\]
Now $E_\ii\subset 3 \cdot 4 \lalpha^{-1}I$. Since the endpoints of $E_\ii$ are in the limit set, for each $\ii\in\Phi(I)$ there are $\kk_l,\kk_r\in\Gamma_M$ so that $E_{\ii\kk_l}$ contains the left endpoint of $E_\ii$ and $E_{\ii\kk_r}$ contains the right endpoint of $E_\ii$. By \ref{M3} and definition of $M$, $\diam E_{\ii\kk_l},\diam E_{\ii\kk_r} \leq \frac{\alpha}{4}\diam E_\ii\leq \diam I$, and by \ref{M4}, $\diam E_\ii \geq 4 \diam I$, leading to $\dist( E_{\ii\kk_l} , E_{\ii\kk_r} ) \geq 2\diam I$, so they both cannot intersect $I$. Choose one that does not intersect and label it $\ii'$. By definition of the measure $\tilde\nu$, we have $\nu([\ii']) \geq p_*^M \nu([\ii])$. Since
\[
 \pi^{-1}I \subset \bigcup_{\ii\in\Phi(I)} [\ii]\setminus [\ii'],
\]
we get that
\begin{align*}
 \nu(I)
 &=
 \tilde\nu( \pi^{-1}I )
 \leq
 \sum_{\ii\in\Phi(I)} \tilde\nu([\ii]\setminus [\ii']) \\
 &\leq
 \sum_{\ii\in\Phi(I)} (1-p_*^M)\tilde\nu[\ii]
 \leq
 (1-p_*^M) \nu( 12\lalpha^{-1}I ).
\end{align*}
 Writing $2^{-\gamma} = (1-p_*^M)$, $N = 12\lalpha^{-1} $, and $I=B(x,r)$
 \[
  2^{\gamma} \nu(B(x,r)) \leq \nu(B(x,N r))
 \]
 for all $x$ and $r$ with $\spt\nu \not\subset B(x,N r)$.
\end{proof}
\begin{corollary}
 \label{cor:lq_dim_Moran}
 Let $\nu$ be a Moran construction measure associated to a Moran construction $\{ E_\iii : \iii \in \Gamma_* \}$ satisfying \ref{M1}--\ref{M5}. Then for all $\eta\in (0,1)$ there exists $\eps=\eps(\eta,\beta,\underline{\alpha})>0$ so that $D(\mu\ast\nu,q) >  D(\mu,q)+\eps$ for any measure $\mu$ with $ D(\mu,q) < 1-\eta$.
\end{corollary}

Let $\Phi = \{\fii_i\}_{i=1}^\kappa$ be a collection of contracting differentiable maps, say from some open neighborhood $V$ of $[0,1]$ to itself and suppose that $\Phi$ satisfies the so called bounded distortion property: there exists $\beta > 1$ so that $|\fii_\ii'(x)| \leq \beta |\fii_\ii'(y)|$ for all $x,y\in V$ and $\ii\in\Gamma_*$. Such an IFS is called self-conformal. From \cite[Lemma 2.1]{KaenmakiRossi2016} and \cite[Example 2.2]{KaenmakiRossi2016} we can see that self-conformal sets, which are not a singleton can always be represented by Moran constructions satisfying \ref{M1}--\ref{M5}. Hence we have the following corollary.
\begin{corollary}
 \label{cor:lq_dim_self_conf}
  If $\{ \fii_i \}_{i=1}^\kappa$ be a self-conformal iterated function system with attractor $E$, and assume that $E$ is not a singleton. Let $\tilde\nu$ be a measure on $\Gamma$ constructed from $(p_\ii)_{\ii\in\Gamma_*}$ satisfying \eqref{eq:conditions-measure}, and denote by $\nu$ the projection of $\tilde\nu$ to $E$. In particular, this holds if $\tilde\nu$ is a Bernoulli measure or a Gibbs measure associated to a H\"{o}lder potential.

  Then given $\eta\in (0,1)$ there is $\e>0$ such that  $D(\mu\ast\nu,q) >  D(\mu,q)+\eps$ for any measure $\mu$ with $ D(\mu,q) < 1-\eta$.
\end{corollary}

\section{Applications to dimensions of sumsets}\label{sec:sets}
From our results on $L^q$ norms of $2^{-m}$-measures, we can easily deduce results about the increase of upper box counting dimension under taking sumsets. Results of this kind were recently studied in \cite{FraserHowroydYu2018}, based on Hochman's  inverse theorem for entropy from \cite{Hochman2014}. As we discussed in the introduction, Hochman's inverse theorem cannot be used to obtain information on the improvement of $L^q$ dimension but, on the other hand, we can recover several results from \cite{FraserHowroydYu2018} from our main results. Along the way, we compare several notions of ``uniform largeness'' that appear in the literature: uniform perfectness (for sets),  positive lower dimension and positive thickness. We also review a beautiful result of S. Astels that appears to be little known.

If $A,B\subset \R$, then their sumset $A+B$ is defined as $\{a+b:a\in A, b\in B\}$. We denote the $n$-fold sum of $A$ by $n A$ (this is \emph{not} multiplying each element of $A$  by $n$). If $A$ is bounded the upper box counting dimension of $A$, denoted by $\udimb A$, is defined as
\[
 \limsup_{m\to\infty}\frac{\log |N_m(A)|}{m}.
\]
We can also define $\udimb A$ for unbounded sets $A$ by taking supremum of $\udimb A'$ over bounded $A'\subset A$. Taking closure does not affect $\udimb$, and therefore we may always assume our sets to be compact.

To pass from measure results to set results, the idea is to consider at each level $m$ the uniform measure $\mu$ on $A^{(m)}$, where $A^{(m)}$ is the $2^{-m}$-discretization of $A$, defined (similarly to the case of measures) as follows:  $k 2^{-m} \in A^{(m)}$ if and only if $A\cap [k 2^{-m},(k+1)2^{-m})\neq\varnothing$. It is then obvious that $\| \mu \|_2^{-2} = N_m(A) = |A^{(m)}|$. From the definition of convolution of measures it follows that that $\spt(\mu\ast\nu) = \spt\mu + \spt\nu$. Recall also that among all probability measures supported on a given $2^{-m}$-set, the uniform measure has the smallest $L^q$ norm.

It follows immediately from the definitions that the support of a $(N,\gamma)$-uniformly perfect measure is a uniformly perfect set, with constant $N$. Moreover, a compact uniformly perfect set supports a uniformly perfect measure. This follows since a compact subset of $\R$ supports a doubling measure(more precisely, there is a constant $C$ so that every closed subset of $\R$ carries a doubling measure whose doubling constant is $C$; see \cite{LuukkainenSaksman1998}), and then Lemma \ref{lem:doubling_to_unif_perf} finishes the deduction. The same result holds also in a more general setting, see
 \cite[Corollary 3.3]{KaenmakiLehrback2017}.

Let $F\subset\R$ and suppose that for $t\geq 0$ there exists a constant $c_t>0$ so that if $0<r<R< \diam(F)$, then for every $x\in F$ at least $c_t(r/R)^{-t}$ balls of radius $r$ are needed to cover $B(x,R)$. The supremum of such $t$ is called the lower dimension of $F$ and we denote it by $\diml F$. In the literature, for example in \cite{KaenmakiLehrbackVuorinen2013}, this has also been called the lower Assouad dimension of $F$.

Let $F\subset\R$ be a compact set. Then the complement of $F$ is a disjoint countable union of open intervals, two of which are unbounded. Denote the collection of the bounded ones by $\OO$, and also add $\varnothing$ to the collection. Set $\Sigma=\{0,1\}$. And let $\Gamma^*$ be the corresponding set of finite words, as in \S\ref{sec:applications}. Set $I_\varnothing$ to be the smallest closed interval containing $F$. Suppose $I_\ii$ has been defined and that $I_\ii$ contains elements of $\OO$. Choose a bounded interval $O_\ii\in\OO$ contained in $I_\ii$ and set $I_{\ii0} \cup O_\ii \cup I_{\ii1} = I_\ii$, where $I_{\ii0} $ and $ I_{\ii1}$ are the obvious intervals. If $I_\ii$ did not contain elements of $\OO$, then set $O_\ii = \varnothing$ and divide $I_\ii$ into closed intervals $I_{\ii0}$ and $I_{\ii1}$ of equal length. Continuing inductively will produce a function $\DD$ from $\Gamma_*$ to $\OO$, by $\DD(\ii) = O_\ii$, and we have that
\[
 F = I_\varnothing \setminus \bigcup_{\ii\in\Gamma_*} O_\ii.
\]
Note also that $\{I_\ii\}_{\ii\in\Gamma*}$ is a Moran construction with limit set $F$. The open intervals $O_\ii$ are called gaps, and the closed intervals $I_\ii$ are called bridges. The labeling $\DD$ is called the derivation of $F$. If $\DD(\ii)\neq\varnothing$ we say that $I_\ii$ splits. If $I_\ii$ splits set
\[
 \tau_\DD(\ii) = \min\left\{ \frac{ \diam(I_{\ii0}) }{ \diam(O_\ii) } , \frac{ \diam(I_{\ii1})}{ \diam(O_\ii) } \right\}
\]
and set $\tau_\DD(\ii)=\infty$ otherwise. We define the thickness of the derivation $\DD$ to be $\tau_\DD = \inf_{\ii\in\Gamma*} \tau_\DD(\ii)$. Finally,  we define the thickness of $F$ to be
\[
 \tau(F) = \sup\{ \tau_\DD : \DD \text{ is a derivation of } F \}.
\]
Note that $\tau(F)\in [0,+\infty]$. If $F$ contains an isolated point, then $\tau(F) = 0$. On the other hand $\tau(F) = \infty$ if and only if $F$ is an interval (see \cite[Lemma 2.1]{Astels2000}). For compact sets, positive thickness, uniform perfectness, and positive lower dimension are equivalent concepts in a quantitative way:

\begin{proposition} \label{prop:equivalence-thickness}
 \label{prop:equvalent}
 Let $C\subset\R$ be compact. Then the following claims hold:
 \begin{enumerate}
  \item\label{enu:uni_to_lower} If $C$ is uniformly perfect with constant $K$, then $\diml C > (\log 2 K +1)^{-1}$.
  \item\label{enu:lower_to_uni} If $\diml C > t$, and the involved constant for $t$ is $c_t>0$, then $C$ is uniformly perfect with constant $K = \sqrt[t]{2/c_t}$.
  \item\label{enu:uni_to_thickness} If $C$ is uniformly perfect with constant $K$, then $\tau(C) \geq K^{-1}$
  \item\label{enu:thickness_to_uni} If $\tau(C) > 2 (K-1)^{-1}$, then $C$ is uniformly perfect with constant $K$.
 \end{enumerate}
\end{proposition}
\begin{proof}
 The parts \eqref{enu:uni_to_lower} and \eqref{enu:lower_to_uni} are in the proof of \cite[Lemma 2.1]{KaenmakiLehrbackVuorinen2013}.

 To prove \eqref{enu:uni_to_thickness}, assume that $C$ is uniformly perfect with constant $K>1$. It is enough to construct a derivation of $C$ with $\tau_\DD \geq K^{-1}$. Let $I_\varnothing$ be the smallest closed interval containing $C$. Assume that $I_\ii$ has been constructed. Assume that $I_\ii$ splits, and choose $O_\ii$ to be the maximal open interval in the complement of $C$ contained in $I_\ii$, and let $I_{\ii0}$ and $I_{\ii1}$ be the obvious left and right intervals - possibly singletons. Since we have always chosen a maximal gap, the gaps on both sides of $I_\ii$ are larger than $O_\ii$. Hence, by uniform perfectness we have that
 \[
  \tau_{\DD}(\ii)
  =
  \min\left\{ \frac{ \diam(I_{\ii0}) }{ \diam(O_\ii) } , \frac{ \diam(I_{\ii1})}{ \diam(O_\ii) } \right\}
  \geq
  \frac{ \diam(I_{\ii j}) }{ K \diam(I_{\ii j})} = \frac{1}{K},
 \]
 for some $j\in\{0,1\}$. If $I_\ii$ did not split, then $\tau_\DD(\ii)=\infty$. All in all, $\tau (C) \geq 1/K$.

 To prove \eqref{enu:thickness_to_uni}, assume that $C$ has thickness $>2 K^{-1}$, but is not uniformly perfect with constant $K$. Then there exists $x\in C$ and $r>0$ so that $C \setminus B(x,r) \neq \varnothing$, but $B(x,Kr)\setminus B(x,r)$ does not meet $C$. Let $\DD$ be any derivation of $C$. Note that $B(x, K r)\setminus B(x,r)$ must be contained in a union of two different gaps $O_\ii$ and $O_\jj$. Without loss of generality, assume that $O_\jj$ is to the left of $x$ and that in the derivation $O_\jj$ is removed before $O_\ii$. Then the bridge $I_{\ii0}$ satisfies $\diam I_{\ii0} \leq 2r$, and on the other hand we have $\diam O_\ii \geq (K-1)r$. Hence for all derivations $\DD$ of $C$ have $\tau_\DD \leq 2/(K-1)$, implying $\tau(C) \leq 2/(K-1)$, which contradicts our assumption.
\end{proof}

For the sake of example, we show how to derive \cite[Theorem 2.1]{FraserHowroydYu2018} from our Proposition \ref{prop:lq_norm_unif_perf}.
\begin{theorem}[{\cite[Theorem 2.1]{FraserHowroydYu2018}}]
 \label{thm:fraser_etal}
 If $\dim_L F_2 > 0$ and $\udimb F_1 < 1$, then $\udimb F_1 < \udimb(F_1 + F_2)$.
\end{theorem}
\begin{proof}
By the usual rescaling and translating, we may assume that $F_1, F_2\subset [0,1)$. For each $m$, let $\mu^{(m)}$ be the uniform probability measure on $F_1^{(m)}$. We have that $\|\mu^{(m)}\|_2^2 = |F_1^{(m)}|^{-1}$ and $\udimb F_1 <  1-\eta$, so $\|\mu^{(m)}\|_2^2 \geq 2^{-(1-\eta) m}$ for all $m$ large enough.

Since we may assume $F_2$ to be compact, and it is uniformly perfect by Proposition \ref{prop:equivalence-thickness}, it supports an $(N,\gamma)$-uniformly perfect measure $\nu$. Thus for each $m$ we have that $\nu^{(m)}$ is an $(2N+1,\gamma)$-uniformly perfect measure on $F_2^{(m)}$. Now the assumptions of Proposition \ref{prop:lq_norm_unif_perf} are satisfied for all $m$ large enough (and the constants $\eta,N$, and $\gamma$ are independent of $m$), so there exists $\eps>0$ so that
\begin{equation}
 \label{eq:lq_boxdimension_improving}
 \| \mu^{(m)} \ast \nu^{(m)} \|_2 \leq 2^{-\eps m} \|\mu^{(m)}\|_2.
\end{equation}
for all $m$ large enough. Using  Cauchy-Schwarz, this yields
\[
2^{\eps m} |F_1^{(m)}| = 2^{-\eps m} \|\mu^{(m)}\|_2^{-2} \le \| (\mu\ast\nu)^{(m)}\|_2^{-2} \leq |(F_1+F_2)^{(m)}|,
\]
always assuming $m$ is large enough. Taking logarithms, dividing by $m$ and taking limsup as $m\to\infty$ we obtain the claim.
\end{proof}
From this result one can easily deduce that if $F$ satisfies $\diml F > 0$, then $\udimb (n F)$ converges to one. The stronger statement that $\diml(n F)\to 1$ was proved in \cite[Theorem 2.6]{FraserHowroydYu2018}, but actually a much stronger result is true: after finitely many steps, the sumset becomes an interval. This follows by work of Astels, in particular from \cite[Theorem 2.4]{Astels2000}, on the sumsets of Cantor sets with thickness bounded from below, which we restate for convenience:
\begin{theorem} \label{thm:Astels}
Let $C_1,\ldots,C_k$ be Cantor sets on $\R$ such that
\[
\sum_{n=1}^k \frac{\tau(C_n)}{\tau(C_n)+1} \ge 1.
\]
Then $C_1+\ldots+C_k$ contains an interval. Moreover, if the largest gap of any of the $C_n$ is smaller or equal than the smallest diameter of any of the $C_n$, then $C_1+\ldots+C_k$ is an interval. In particular, this is the case is all the $C_k$ are the same set $C$ and $\tau(C)\ge 1/(n-1)$.
\end{theorem}

\begin{corollary}
 \label{cor:interval}
 If $(C_n)$ is a sequence of compact uniformly perfect sets in $\R$ with common constant $K$ such that the largest of their gaps is at most the smallest of their diameters, then there exists a finite $n=n(K)$ such that $\sum_{i=1}^n C_i$ is an interval.

 In particular, if a compact set $C\subset\R$ satisfies $\diml C > 0$ (or, equivalently, is uniformly perfect or has positive thickness), then there exists a finite $n$ such that $n C$ is an interval. Thus, this is the case if $C$ is an $\alpha$-regular set with $\alpha>0$.
\end{corollary}

\begin{remark}
While Theorem \ref{thm:fraser_etal} does not seem to follow directly from Corollary \ref{cor:interval}, the version of Theorem \ref{thm:fraser_etal} in which upper box dimension is replaced by lower box dimension does follow from Corollary \ref{cor:interval} together with the version of the Pl\"{u}nnecke inequalities for box dimension given in \cite[Proposition 1]{ShmerkinSchmeling2010}.
\end{remark}

\begin{remark}
\label{rem:sets:_and_measures}
If $(C_n)$ is a sequence of compact sets with $\diml C_n > \delta > 0$ for all $n\in\N$ then we do not know whether $\udimb( \sum_{i=1}^n C_i) $ necessarily converges to $1$. If $\diml C_n > \delta$ for all $n\in\N$ but the constant  in the definition of uniform perfectness blows up very rapidly, then we can not apply the results of Astels or use Proposition  \ref{prop:lq_norm_unif_perf}. On the other hand, we have not been able to find a counterexample.
\end{remark}

\begin{remark}
Note that one cannot hope to get a result similar to Corollary \ref{cor:interval} for measures. If $C$ is the familiar middle $\tfrac13$-Cantor set and $\mu$ the uniform Cantor-Lebesgue measure on $C$, then it holds that already $C+C$ is an interval, but $\mu^{\ast n}$ is purely singular to Lebesgue measure for every $n$. Indeed, it is well known that the Fourier transform $\widehat{\mu}(\xi)$ does not tend to $0$ as $\xi\to\infty$ (this follows since $\widehat{\mu}(3\xi)=\widehat{\mu}(\xi)$ for all $\xi$), so the same holds for $\mu^{\ast n}$ by the convolution formula, and by the Riemann-Lebesgue Lemma $\mu^{\ast n}$ cannot be absolutely continuous. But $\mu^{\ast n}$ is still self-similar, so by the well-known law of pure type it must be purely singular.
\end{remark}


\bibliographystyle{plain}
\bibliography{improving_refs}

\end{document}